\documentclass{article}

\usepackage{amsmath,amssymb,amsfonts,bbm,color,bm,nicefrac,
graphicx,enumerate}
\usepackage{amsthm}
\usepackage{epstopdf}
\usepackage{geometry}
\usepackage{caption}
\usepackage[toc,page,title]{appendix}
\usepackage{authblk}

\newtheorem{thm}{Theorem}[section]
\newtheorem{theorem}{Theorem}[section]

\newtheorem{lemma}[thm]{Lemma}

\newtheorem{proposition}[thm]{Proposition}

\newtheorem{definition}{Definition}[section]
\newtheorem{remark}[thm]{Remark}

\newcommand{\comment}[1]{}

\begin{document}
\nocite{*}
\title{Coarsening Dynamics on $\mathbb{Z}^d$ with Frozen Vertices}

\author{M. Damron$^{1}$, %\thanks{The research of M. D. is supported by NSF grant DMS-1419230.},
       S. M. Eckner $^{2}$, 
       %\thanks{The research of S. E. is supported in part by NSF grants OISE-0730136 and DMS-1007524.},
       H. Kogan$^{2}$,
%$\thanks{},
       C. M. Newman$^{3}$, 
       %\thanks{The research of C. M. N. is supported in part by NSF grants OISE-0730136 and DMS-1007524},
       V. Sidoravicius$^{4}$
 %$\thanks{}
       \newline
       \\
       $^1$ Department of Mathematics, Indiana University, Bloomington \\
       $^2$ Courant Institute of Mathematical Sciences \\
       $^3$ Courant Institute of Mathematical Sciences and NYU--Shanghai \\
      % $^5$ Department of Mathematics, University of California, Irvine \\
       $^4$ IMPA, Rio de Janeiro, Brazil
       }
\maketitle

\begin{abstract}
We study %continuous time
Markov processes in which $\pm 1$-valued random variables
$\sigma_x(t), x\in \mathbb{Z}^d$, update by taking the value of a majority of their
nearest neighbors or else tossing a fair coin in case of a tie. In the presence of a random
environment of frozen plus (resp., minus) vertices with density $\rho^+$ (resp., $\rho^-$),
we study the prevalence of vertices that are (eventually) fixed plus or fixed minus or
flippers (changing forever). Our main results are that, for $\rho^+ >0$ and $\rho^- =0$,
all sites are fixed plus, while for $\rho^+ >0$ and $\rho^-$ very small (compared to
$\rho^+$), the fixed minus and flippers together do not percolate. We also obtain some results
for deterministic placement of frozen vertices.
\end{abstract}

%%%%%%%%%%%%%%%%%%%%%%%%%%%%%%%%%%%%%%%%%%%%%%%%%%%%%%%%%%%%%%%%%%%%%%%

\section{Introduction}\label{sec:Introduction}

In this work we study the long time behavior of
Markov processes, primarily in continuous time,
whose states assign either $+1$ or $-1$, called a spin value, to each
vertex of the $d$-dimensional lattice $\mathbb{Z}^d$. We will discuss two types of
processes, a much studied one denoted $\sigma(t)= (\sigma_x (t):x\in \mathbb{Z}^d)$ and then
a modified one, denoted $\sigma'(t)$, in which some vertices are ``frozen'' -- that is, their
spin values are not allowed to change.

Before giving complete definitions of $\sigma(t)$ and $\sigma'(t)$, we give brief
descriptions and motivations. $\sigma(t)$ has an energy-lowering dynamics with energy of
the form $ -\sum^* \sigma_x \sigma_y$, where $\sum^*$ denotes the sum over nearest neighbor
pairs of vertices. Here, energy is lowered at the update of $\sigma_x$ if its value is
changed to agree with a strict majority of neighbors. The modified process $\sigma'(t)$
basically corresponds to $\sigma(t)$ in a random environment where randomly or deterministically selected
vertices are frozen from time zero, some plus and some minus.

There are two distinct motivations for studying $\sigma'$. One, explained in more detail
below, comes from the usual $\sigma$ process, with random initial state, but on a slab,
say $\mathbb{Z}^2 \times \{0, \ldots, k-1 \}$, rather than on $\mathbb{Z}^d$. Here,
random rectangular regions of the form $R \times \{0, \ldots, k-1 \}$ are fixed
(thus effectively frozen) from time zero if they start with constant spin value. A second
motivation comes from the energy-lowering dynamics of random-field models with energy
$-\sum^* \sigma_x \sigma_y -\sum_x h_x \sigma_x$ where the $h_x$'s are i.i.d. variables
with common distribution $\rho$. Suppose, for example, that $\rho = \rho^+ \delta_H +
\rho^- \delta_{-H} + (1-\rho^+ -\rho^-)\delta_0$ with $H>2d$,
where $\delta_r$ denotes the unit point measure at $r$. The reader can then check
that at vertex $x$ with $h_x=H$ (resp., $-H$), either from time zero or else after the first
update at $x$, $\sigma_x (t)$ will be fixed plus (resp., minus) and thus
effectively frozen. We
proceed now to define our two processes, $\sigma$ and $\sigma'$, and review some known
results about~$\sigma$.

\section{The two processes}

\subsection{The process $\sigma$}
The stochastic process $\sigma(t) = \sigma (t, \omega)$, where
%where $\omega$ is an element of a probability space $\Omega$ which will be
%defined later,
$\sigma_x (t)$ denotes the value of the spin
at vertex $x \in \mathbb{Z}^d$ at time $t \geq 0$, starts from a random initial configuration
$\sigma (0) =\{\sigma_x (0)\}_{x \in \mathbb{Z}^d}$, drawn from an independent Bernoulli
product measure

\begin{equation}\label{mu}
 \mu_{\theta} (\sigma_x (0) = + 1) = \theta = 1 - \mu_{\theta} (\sigma_x (0)
   = - 1).
\end{equation}

\noindent
The system evolves in continuous time according to an agreement inducing
dynamics: at rate 1, each vertex changes its
value if it disagrees with more than half of its neighbors, and decides its
spin value by tossing a fair coin in the event of a tie. This process
corresponds to the zero-temperature limit of Glauber dynamics
for a stochastic Ising model with ferromagnetic nearest-neighbor interactions
and no external magnetic field (see, for example, \cite{NNS} or \cite{KRB}), having
Hamiltonian (energy function)

\begin{equation}\label{hamiltonian}
 \mathcal{H} = - \sum_{\{x, y\}: \|x - y\| = 1} \sigma_x \sigma_y ,
\end{equation}

\noindent
where $\|x\|$ denotes
the Euclidean norm of $x$.

More precisely, let $\mathcal{S}$ be the state space of configurations $\sigma$,
i.e., $\mathcal{S} =\{- 1, 1\}^{\mathbb{Z}^d}$. The continuous time dynamics can be defined by
means of independent, rate 1 Poisson processes (clocks), one assigned to each vertex $x$. If
the clock at vertex $x$ rings at time $t$ and the change in energy

\begin{equation*}
 \Delta \mathcal{H}_x (\sigma) = 2 \sum_{y : \|x - y\|= 1} \sigma_x
   \sigma_y
\end{equation*}

\noindent
is negative (respectively, positive), a spin flip
(that is, a change of $\sigma_x$) is done with probability 1
(respectively 0). To resolve the case of ties when $\Delta \mathcal{H}_x (\sigma) = 0$,
each clock ring is associated to a fair coin toss and a spin flip is done with
probability 1/2, or equivalently $\sigma_x$ is made to be $+1$ (respectively, $-1$) if
the coin toss comes up heads (resp., tails).
Let $\mathbb{P}_{\mathrm{dyn}}$ be the probability measure for the
realization of the dynamics (clock rings and tie-breaking
coin tosses), and denote by $\mathbb{P}_\theta=
\mu_{\theta} \times \mathbb{P}_{\mathrm{dyn}}$ the joint probability measure on
the space $\Omega$ of initial configurations $\sigma (0)$ and realizations of
the dynamics; an element of $\Omega$ is denoted~$\omega$.

This process has been studied extensively in the physics and mathematics
literature -- primarily on graphs such as the hyperlattice $\mathbb{Z}^d$ and
the homogeneous tree of degree $K$, $\mathbb{T}_K$. A physical motivation,
which corresponds to the symmetric initial spin configuration,
is the behavior of a magnetic system following a deep quench. A deep
quench occurs when a system that has reached equilibrium at an initial high temperature
$T_1$ is instantaneously subjected to a very low temperature $T_2$. Here
we take $T_1 =\infty$ and $T_2 =0$. For references on
this and related problems see, for example, \cite{NNS} or \cite{KRB}.
The main focus in the study of this model is the formation and evolution of
boundaries delimiting same spin cluster domains. These domains shrink or grow or split or coalesce
as their boundaries evolve. This model is often referred to as
a model of \emph{domain coarsening}.

A fundamental question is whether the system
has a limiting configuration, or equivalently does every vertex
eventually stop flipping? Whether

\begin{equation}\label{sigmainfinity}
\lim_{t\rightarrow\infty}\sigma_x(t)
\end{equation}

\noindent
exists for almost every initial configuration, realization of the dynamics and
for all $x\in \mathbb{Z}^d$ depends on $\theta$
and on the dimension $d$. We refer to the existence of the limit~\eqref{sigmainfinity} as {\bf fixation} at $x$.

Nanda, Newman and Stein \cite{NNS}
investigated this question when $d=2$ and $\theta=\frac{1}{2}$
and found that the limit does not exist; that is, every vertex
flips infinitely often. Their work extended an old result
of Arratia \cite{A}, who showed the same happens on $\mathbb{Z}$
(for $\theta\neq 0$ or $1$). It is an open problem to determine what happens
for $d\geq 3$ and $\theta = 1/2$.
One important consequence of the methods of \cite{NNS} is that if each
vertex of the graph has odd degree (for example,
on $\mathbb{T}_K$ for $K$ odd), then $\sigma_x (\infty)$ does exist for almost every initial configuration,
realization of the dynamics and every vertex $x$.

Another question of interest is whether sufficient bias in the initial
configuration leads the system to reach consensus in the limit. That is, does
there exist
$\theta_{\ast}\in(0, 1)$, such that for $\theta \geq \theta_{\ast}$,

\begin{equation}\label{fixation}
\forall x \in G, \mathbb{P}_{\theta} (\exists T = T (\sigma (0), \omega, x) < \infty
\text{ so that } \sigma_x (t) = + 1 \text{ for } t \geq T) = 1?
\end{equation}

\noindent
We will refer to \eqref{fixation} as {\bf fixation to consensus}. It was conjectured by
Liggett \cite{L} that fixation to consensus holds for all $\theta > \frac{1}{2}$.
Fontes, Schonmann and Sidoravicius \cite{FSS} proved fixation to consensus
for all $d\geq 2$ and $\theta_{\ast}$ strictly less but very close to 1. On $\mathbb{T}_3$, however, Howard \cite{Howard} showed that for some $\theta>1/2$, the system does not fixate to $+1$ consensus.

\subsection{$\sigma$ on slabs}

In \cite{DKNS1} and \cite{DKNS2}, Damron, Kogan, Newman and Sidoravicius studied
coarsening started from a configuration sampled from $\mu_\theta$ on
two-dimensional slabs of finite width $k$
with free boundary conditions, which we denote as $\text{Slab}_k$. These are
graphs with vertex set $\mathbb{Z}^2 \times \{0, 1, \ldots, k-1\}$ ($k \geq 2$)
and edge set $\mathcal{E}_k =\{\{x, y\}: \|x-y\|_1=1\}$.
Their work was motivated by the question of whether there are vertices
that fixate for $d\geq 3$ (and for which values of $d$). It
has been suggested by computational results of Spirin, Krapivsky and
Redner \cite{SKR} that some vertices do indeed
fixate.

The work in  \cite{DKNS1} and \cite{DKNS2} on slabs highlights the
possible diferences in long time behavior between $\mathbb{Z}^2$
%and what is believed to
%occur on
and $\mathbb{Z}^3$. % based on numerical results -- see \cite{NS}.
The authors
showed that if $k=2$ the system fixates with both free and periodic boundary
conditions; if $k=3$ with periodic boundary conditions the system also fixates;
for all $k\geq 4$ with periodic boundary conditions some vertices fixate for
large times and some do not, and the same result holds for all $k\geq 3$ with
free boundary conditions. %That is, for all $k\geq 3$ with free boundary conditions and
%for all $k\geq 4$ with periodic boundary conditions, with positive probability,
%there exist
We call vertices which change spin sign forever flippers.
One question which remains open, is whether the set of flippers
percolates (contains an infinite component).

%Let us consider the  case $k = 3$. In this case the system has both
%vertices that fixate (in fact some vertices are already fixed
%starting at $t=0$) and vertices that change spin forever, which we call
%flipping vertices or flippers. In particular, pillar-like sets of the form
%\begin{equation}
% R=\{(x, y), \ldots, (x+L, y) \} \times \{(x, y)\ldots (x, y+M) \} \times
%\{0, 1, 2 \},
%\end{equation}

%\noindent
%for $(x, y)\in \mathbb{Z}^2$ and $L, M\geq 1$ are already fixed at $t=0$.
%An interesting open question is whether the set of flipping vertices percolates
%rather than forming only finite components surrounded by fixed vertices.

On $\text{Slab}_3$, each $v\in \mathbb{Z}^2 \times\{0\}$ or
$v \in \mathbb{Z}^2 \times\{2\}$ has five
neighbors, so by a variant of the arguments of Nanda, Newman,
Stein \cite{NNS}, $v$ fixates almost surely. Therefore on this graph,
flippers can only exist in $\mathbb{Z}^2 \times \{1\}$. On the other hand,
if the initial configuration
on $\text{Slab}_3$ is chosen according to a symmetric Bernoulli product measure, then by
the Ergodic Theorem there are at $t=0$ infinitely many pillar-like same-spin
formations (say, $2\times 2 \times 3$ blocks) that are stable under the
dynamics. These pillars are analogous to frozen vertices on $\mathbb{Z}^2$
in our new process $\sigma'$. A
general version of this new process
on $\mathbb{Z}^d$ is presented in the following section.

\subsection{The process $\sigma'$}

In this subsection we define a new stochastic process on $\mathbb{Z}^d$,
which we denote by~$\sigma'(t)$, in which some vertices are frozen for all time and
the others are not. There are two basic versions of $\sigma'$, which we will call
{\bf disordered} and {\bf engineered}, according to whether the
frozen vertices are chosen randomly or deterministically.
The initial configuration of the disordered $\sigma'$
will be assigned as follows. Fix $\rho^+, \rho^- \geq 0$ with $\rho^+ + \rho^- \leq 1$
and pick three types of vertices (frozen plus, frozen minus and unfrozen) by i.i.d.
choices with respective probabilities $\rho^+, \rho^-$ and $1-(\rho^+ + \rho^-)$.
Once the frozen vertices have been chosen and assigned a spin value, the non-frozen
vertices will be assigned spin values arbitrarily. (In other words, the theorems we prove will be valid for all choices of such spin values.)
We will denote by $\mathbb{P}=\mathbb{P}_{\rho^+, \rho^-} \times \mathbb{P}_{\text{dyn}}$ the
overall probability measure where $\mathbb{P}_{\rho^+, \rho^-}$ is the distribution for
the assignment of frozen plus and minus vertices and $\mathbb{P}_{\text{dyn}}$ is the
distribution of the following dynamics for $\sigma'(t)$.
The continuous time dynamics is defined similarly to that of the $\sigma(t)$ process.
Vertices are assigned independent, rate 1 Poisson clock processes and
tie-breaking fair coins, and flip sign to agree with a majority of their neighbors.
Frozen vertices, however, never flip regardless of the configuration
of their neighbors. In the engineered $\sigma'$ the frozen vertices are chosen
deterministically while the non-frozen vertices are assigned values at time zero
by i.i.d. choices with respective probabilities $\theta$ and $1-\theta$ for
$+1$ and $-1.$

As usual, we are interested in the long term behavior of this model depending
on the dimension $d$, the densities of frozen vertices, $\rho^+, \rho^-$ (or in the
engineered version, the choice of frozen vertices), and
the initial configuration of vertices, which we denote by
$\sigma'(0)$. $\sigma'(0)$ may be regarded as an element of $\{-1,+1\}^{\mathbb{Z}^d}$ even though the frozen vertex spins are instantaneously replaced (at time $t=0+$) by their frozen values. Note that the $\sigma'$ processes for all possible choices of $\sigma'(0)$ are coupled on a single probability space. When $\rho^+>0 $ and $\rho^->0$, almost surely there exist
flippers. To see this, consider the following configuration for the case $d=2$,
which has probability $(1-\rho^+-\rho^-)(\rho^+)^2(\rho^-)^2$: the vertex labeled $v$ in
Figure~\ref{fig:flippervertex} below is not frozen, but has two frozen
neighbors of spin $+1$ and two frozen neighbors of spin $-1$, and thus flips
infinitely often. Similar flippers, as well as more complicated clusters of flippers,
occur for any $d$.

\begin{figure}[h]
\centering
\includegraphics{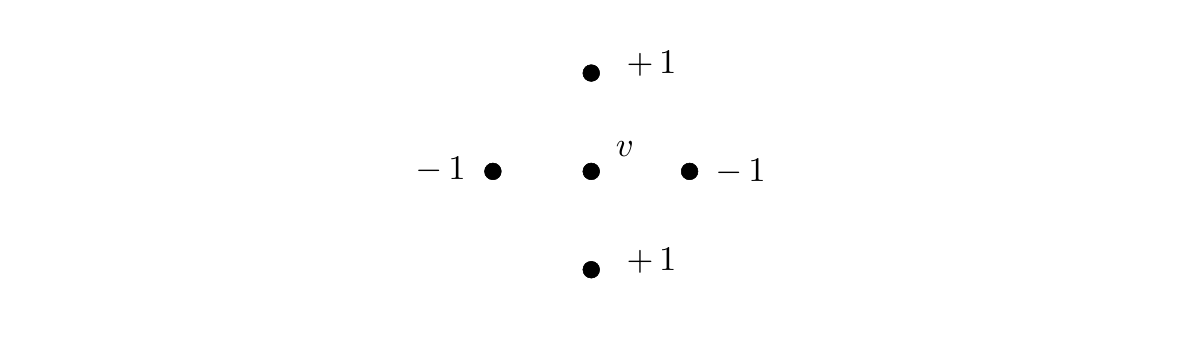}
\caption{A vertex that flips infinitely often.}
\label{fig:flippervertex}
\end{figure}

We conclude that $\sigma_v (t)$ need not have a limit as $t \rightarrow \infty$. So we are interested in whether the flippers percolate. Of course, one may study percolation of any of the three types of vertices (fixed plus, fixed minus or flippers) or of the union of two of the
three. Theorem \ref{thm_frozen_+-} below answers a question in this direction.

\section{Main results}\label{sec:theorems}

Our first two theorems concern the disordered version of $\sigma'$ and the last concerns engineered versions.
The first theorem is a fixation to consensus result
for the case of positive initial density of frozen $+1$'s and zero initial
density of frozen $-1$'s. The second theorem is a more general result in which
both $\rho^+, \rho^- >0$, but we require $\rho^-$ to be much smaller than~$\rho^+$.
For this more general case we obtain that the set of flippers together with
(eventually) fixed vertices of spin $-1$ does not percolate.

\begin{theorem}\label{thm_frozen_+d}
 Consider the disordered stochastic process $\sigma'$ on $\mathbb{Z}^d$
 for any $d$ and any $\rho^+ >0$, with $ \rho^- = 0$. 
 Then $\mathbb{P}($the system fixates to~$+1$ consensus for any $\sigma'(0))=1$.
\end{theorem}

\begin{theorem}\label{thm_frozen_+-}
 Consider the disordered stochastic process $\sigma'$ on $\mathbb{Z}^d$
 for any $d$ and any $\rho^+ >0$, with $\rho^- >0 $ sufficiently
 small (depending on $\rho^+$ and $d$). For an initial configuration $\sigma'(0)$, denote by $\mathcal{C}(\sigma'(0))$ the collection of (eventually) fixed $-1$ vertices and flippers. Then $\mathcal{C}$ does not percolate:
 \[
 \mathbb{P}(\mathcal{C}(\sigma'(0)) \text{ contains
 an infinite component for some }\sigma'(0))=0.
 \]
\end{theorem}

\begin{remark}\label{remarkmodify}
Our main results, with essentially the same proofs, remain valid when the process
$\sigma'$ is modified in various ways. For example, the rules for breaking ties
can be modified as long as there is strictly positive probability to
update to $+1$. Also the process can evolve
in discrete time with synchronous updating. Another modification is to replace
the two spin
values, $\{-1,+1\}$, by $q$ values, say $\{1,2,\dots,q\}$, as long as
updates respect a majority
agreement of neighbors on one value.
\end{remark}

The next theorem concerns engineered versions of $\sigma'$ in which the frozen
vertices are all $+1$ and chosen deterministically while the other spin values at time zero are
i.i.d. with probability $\theta$ to be $+1$. Since in the disordered $\sigma'$ there are
infinitely many frozen vertices, it's natural to consider choosing infinitely many frozen
$+1$ vertices forming a lower dimensional subset of $\mathbb{Z}^d$, such as
$\mathbb{Z}^{d-k} \times \emptyset_k$ with $k, d-k \geq 1$, where $\emptyset_k$ denotes
the origin in $\mathbb{Z}^k$.

Although we have no results to report for any of these situations, the next theorem
concerns a slab approximation to the $k=1$ case where all the vertices in a codimension
one hyperplane are frozen to $+1$. Note that the frozen hyperplane separates $\mathbb{Z}^d$
into two graphs isomorphic to $\mathbb{Z}^{d-1} \times \{0, 1, 2, \ldots \}$ with
$\mathbb{Z}^{d-1}\times \{0\}$ frozen to $+1$, that evolve independently of each other.

\begin{theorem}\label{thm_slab}
 Consider the engineered stochastic process $\sigma'$ on $\mathbb{Z}^{d-1}\times \{0, 1, 2, \ldots, K \}$,
 for $d\geq~2$ and $K\geq 1$, with all spins on $\mathbb{Z}^{d-1} \times \{0\}$ frozen to $+1$
 and $\sigma'(0)$ for other vertices chosen from $\mu_\theta$. Then for any $\theta>0$, with probability $1$ the system fixates to $+1$ consensus.

\end{theorem}

\section{Proofs}\label{sec:proofs}

In this section we give the proofs of the theorems stated in
Section~\ref{sec:theorems}. Let $B_L=[-L, L]^d$ be the cubic box of side length $2L+1$
centered at the origin and $B_L(x)=x + [-L, L]^d$ be the translated box centered at
$x\in \mathbb{Z}^d$.

\subsection{Bootstrap percolation}
Following Fontes, Schonmann and Sidoravicius \cite[Section~2]{FSS}, we describe the
bootstrap percolation process that assigns
configurations $\{u, s\}^{\mathbb{Z}^d}$ to a subset of $\mathbb{Z}^d$; here $u$
represents an \emph{unstable} spin and $s$ represents a \emph{stable} spin at a vertex.

\begin{definition}
The $d$-dimensional $(u\rightarrow s)$ bootstrap percolation process with threshold
$\gamma$, defined in a finite or infinite volume $\Lambda \subseteq \mathbb{Z}^d$,
starting from the initial configuration $\eta(0) \in \{u, s\}^\Lambda$ is a cellular
automaton which evolves in discrete time $t=0, 1, 2, \ldots$ such that at each time
unit $t\geq 1$ the current configuration is updated according to the following rules.
For each $x\in \Lambda$,
\begin{enumerate}
 \item If $\eta_{x}(t-1)=s$, then $\eta_x(t)=s$.
 \item If $\eta_{x}(t-1)=u$, and at time $t-1$ the vertex $x$ has at least $\gamma$
neighbors in $\Lambda$ in state $s$, then $\eta_x(t)=s$; otherwise, the spin at vertex $x$
remains unchanged; that is, $\eta_x(t)=u$.
\end{enumerate}
\end{definition}

We will consider this process with threshold $\gamma = d$, as its evolution is close to
our coarsening dynamics, and assume the initial configuration to be chosen from an
independent Bernoulli product measure $P(\eta_x(0)=s)=p$, for $p$ small, 
on $\Lambda = \mathbb{Z}^d$. Note that by monotonicity of the dynamics, each 
$\eta_x(t)$ has a limit as $t\to\infty$.

\begin{definition}
A configuration $\eta\in \{u, s\}^\Lambda$ \textbf{internally spans} a region
$B_L(x) \subset \Lambda$, if the bootstrap percolation restricted to $B_L(x)$,
started from %$\eta_0= \eta |_{B_L}$, ends up with all vertices of $B_L$ in state $s$. We
$\eta |_{B_L}$, ends up with all vertices of $B_L$ in state $s$. We
will denote %$\eta |_{B_L}$ by $\eta_L$.
by $\eta_L$ the subset $\{x \in B_L: \, \eta_x=s\}$ for such an $\eta$ and
call it \textbf{spanning}.
\end{definition}

The following proposition, an immediate consequence of results of Schonmann \cite{S},
provides a key ingredient for our proofs.

\begin{proposition}\label{schonmann}
\emph{[Schonmann]}
If $p>0$, then
\begin{equation*}
 \lim_{L \rightarrow \infty} P(B_L \text{ is internally spanned}) = 1.
\end{equation*}
\end{proposition}

\begin{remark}
In \cite{S} a variation of bootstrap percolation with threshold $\gamma=d$, called the% \\ \underline{modified basic model},
\\ {\bf {modified basic model}},
is considered. Here, rule 2 is replaced by one which requires at least one neighbor in each of the $d$ coordinate directions from $x$
to have value  $s$ in order that $x$ change from $u$ to $s$. For the modified basic model, the analogue of Proposition~\ref{schonmann}
remains valid -- see \cite[Proposition~3.2]{S}. This will be used in the proof of Theorem~\ref{thm_slab}.
\end{remark}

\subsection{Preliminary lemmas}
We will make a comparison between the bootstrap percolation process on $\mathbb{Z}^d$ and
our process $\sigma'$ by mapping frozen plus spins to stable spins $s$, and
all other spins to unstable spins $u$. In fact, we will compare only on finite regions which do not contain frozen minus vertices, so it is unimportant that these vertices are mapped to $u$.
%The frozen minus spin is
%mapped to $u$ as, even when $\rho^- >0$, we will only bootstrap finite regions with no
%frozen minus vertices. Indeed, the following definitions will be applied in Lemmas
%\ref{lemmaboostrapweak} and \ref{lemmabootstrap} only to cubes with no frozen minus
%vertices to conclude that at some time $t$ all spins in the cube will be~$+1$.
We say that a region $B_L (x)$ contains a \emph{spanning subset} of frozen plus
vertices if the configuration obtained
by the above mapping internally spans $B_L(x)$.

\begin{definition}\label{Mcaptured}
$B_L(x)$ is \textbf{entrapped} if it contains a spanning subset of frozen
plus vertices. It is \textbf{captured} if it is entrapped and all $2^d$ corners are frozen
plus. It is \textbf{$M$-captured} ($M\in \{0, 1, \ldots, L\}$) if it is entrapped and
for each of the $2^d$ corners $C^i, (i=1, \ldots, 2^d)$, and each coordinate direction
$j=1, \ldots, d$ there is a frozen plus vertex within $B_L(x)$ of the form
$C^{i, j}=C^i+m e^j$ with $|m| \leq M$
(where $e^1, \ldots, e^d$ are the standard basis vectors of $\mathbb{R}^d$) -- see
Figure~\ref{fig:Mcaptured}. Note that captured is the same as 0-captured.
\end{definition}

\begin{figure}[h]
\centering
\includegraphics{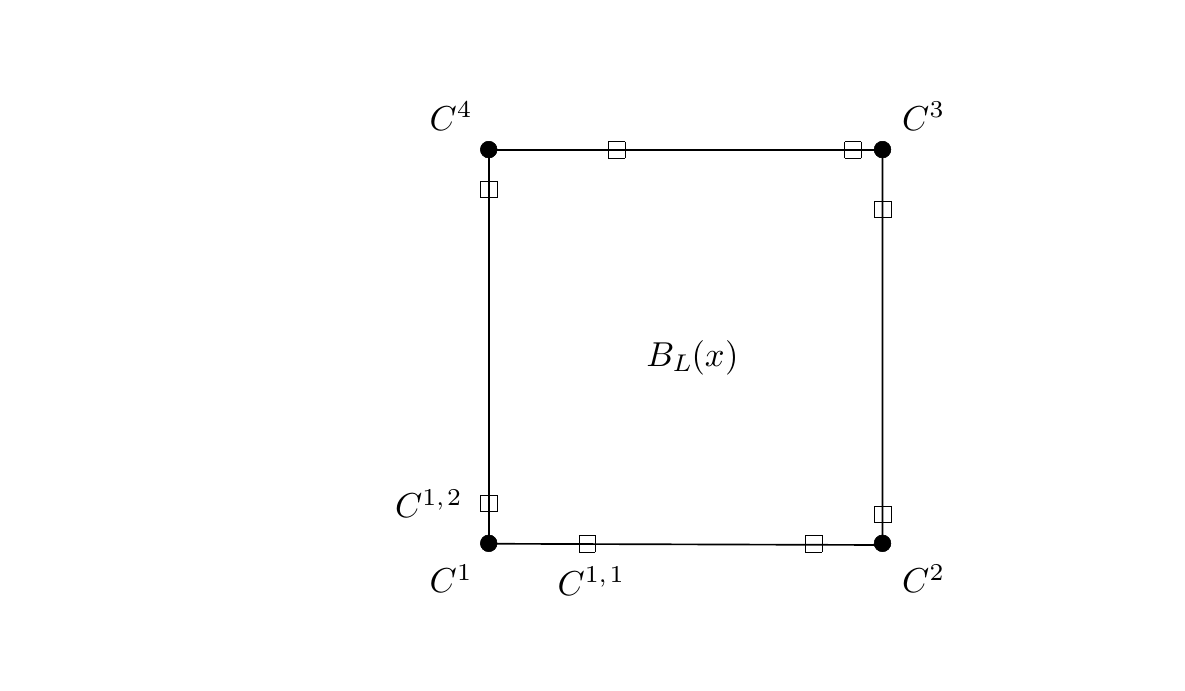}
\hspace{.75in}
\caption{$B_L(x)$ is M-captured.}
\label{fig:Mcaptured}
\end{figure}

The notion of $M$-captured will be used (in Lemma
\ref{lemmaB[M]}) to guarantee that with high probability most of the vertices in the box
$B_L$ (at least for $1 \ll M \ll L$) will fixate to +1. Note that, although in the above
definition $M = L$ is allowed, in Lemmas \ref{lemmaprobMgood} and \ref{lemmaB[M]} we
require $M <L$. Lemma \ref{lemmaMcaptured}, though, will allow us to 
choose $1 \ll M \ll L$ in the final proof, although we do not make use of this.
Lemmas \ref{lemmaboostrapweak}, \ref{lemmabootstrap} and \ref{lemmacaptured}, presented
next, will be used in the proof of Theorem \ref{thm_frozen_+d}.

\begin{lemma}\label{lemmaboostrapweak}
Given $L$ and a spanning subset $\eta_L$ of $B_L$,
consider the $\sigma'$ process in $\mathbb{Z}^d$ with initial spins in $\eta_L$ taken as
frozen plus and all others in $\mathbb{Z}^d \setminus \eta_L$ taken as minus but not frozen.
Then
\begin{equation*}
 \mathbb{P} (\text{for some } t\in[0, 1], \sigma'(t)|_{B_L} \equiv +1) >0.
\end{equation*}
\end{lemma}

\begin{proof}
Since $\eta_L$ is a spanning
subset of $B_L$, the corresponding bootstrap percolation process occupies all vertices of $B_L$ in a finite number of steps. Since the
threshold $\gamma=d$, we can, with a small but positive probability, arrange the clock
rings for $t\in (0, 1)$ and tie-breaking coin tosses of
the coarsening dynamics to mimic the dynamics of bootstrap percolation (in a much longer
discrete time interval). Thus $\sigma'(t)|_{B_L} \equiv +1$ for some $t\in [0, 1]$
with positive probability.
\end{proof}

The next lemma strengthens the last one by showing that,
if we allow the process to run until a large time, then with probability
close to one all the vertices of $B_L$ will flip to $+1$ before that time. Let $Span_L$ be the event that the frozen plus vertices span $B_L$ and there are no frozen minus vertices in $B_L$.

\begin{lemma}\label{lemmabootstrap}
Given $L <\infty$,
\begin{equation*}
 \lim_{T \to \infty} \mathbb{P} \left(\exists t\in[0, T] \text{ such that }  \sigma'(t)|_{B_L} \equiv +1 \text{ for all }\sigma'(0) \mid Span_L \right) = 1.
\end{equation*}
\end{lemma}

\begin{proof}
Pick a maximal spanning subset $\eta_L$ (in some deterministic ordering of subsets) of frozen plus vertices of $B_L$. By
Lemma \ref{lemmaboostrapweak} and attractiveness (monotonicity of the dynamics), there is an $\epsilon' >0$ such that,
for any $m$ and $\sigma'(m)$ (consistent with the frozen vertices in $B_L$),

\begin{equation*}
 \mathbb{P} (\text{for some } t\in[m, m+1],
\sigma'(t)|_{B_L}  \equiv +1 | \sigma'(m)) \geq \epsilon'.
\end{equation*}

\noindent
Given $\epsilon>0$, let $T_L \geq \frac{\log \epsilon}{\log (1-\epsilon')}$ be an integer and apply repeatedly  the last inequality to
time intervals of the form $[m, m+1]$ for integers $m \in [0,T_L)$:

\[
 \mathbb{P} (\text{at some time } t\in[0, T_L], \sigma'|_{B_L} \equiv +1)  \geq
 1- (1-\epsilon')^{T_L}   > 1- \epsilon.
\]
\end{proof}

\begin{lemma}\label{lemmacaptured}
 If $\rho^+ >0$, then
\begin{equation*}
 \lim_{L \rightarrow \infty}\mathbb{P}(B_l \text{ is captured for some } l \leq L)=1.
\end{equation*}
\end{lemma}

\begin{proof}
By Proposition~\ref{schonmann} with $p=\rho^+$, pick a sequence of increasing box sizes $L_i \in \mathbb{N}$ such that
$L_1 < L_2 < \ldots$ and
\begin{equation*}
 \mathbb{P} (B_{L_i} \text{ is not entrapped}) < \frac{1}{i^2}.
\end{equation*}

\noindent
By the Borel-Cantelli Lemma, almost surely, all but finitely many boxes $B_{L_i}$ are entrapped.
Now the probability that each box $B_{L_i}$ has all corners frozen plus equals

\begin{equation*}
 \mathbb{P} (B_{L_i} \text{ has all corners frozen plus}) = (\rho^+)^{2^d} >0.
\end{equation*}

\noindent
By the Law of Large Numbers, almost surely infinitely many boxes $B_{L_i}$ have this property.
Combining these statements, almost surely infinitely many boxes $B_{L_i}$ are captured, which implies the conclusion of the lemma.
\end{proof}

The remaining lemmas will be used in the proof of
Theorem \ref{thm_frozen_+-}.

\begin{lemma}\label{lemmaMcaptured}
 If $\rho^+ > 0$, then
\begin{equation*}
 \lim_{M, L \rightarrow \infty} \mathbb{P} (B_L \text{ is $M$-captured}) =1,
\end{equation*}

\noindent
where $M, L$ tend to infinity with no restriction other than $M \leq L$.
\end{lemma}

\begin{proof}
By Proposition \ref{schonmann}, as in the proof of Lemma \ref{lemmacaptured},
\begin{equation*}
 \lim_{L \rightarrow \infty} \mathbb{P} (B_L \text{ is entrapped}) =1.
\end{equation*}

\noindent
Now for any fixed $L$ and $M\leq L$, let $A_{L, M}$ denote the event that there exist frozen
plus spins within distance $M$ from each of the $2^d$ corners of $B_L$ in every one
of the $d$ coordinate directions,
as in Definition \ref{Mcaptured}. Thus the event that $B_L$ is $M$-captured is the
intersection of $A_{L, M}$ with the event that $B_L$ is entrapped.
The probability of the event that any specific collection of $M$ vertices contains no
frozen plus spins is $(1-\rho^+)^M$, so

\begin{equation*}
 \mathbb{P} (A_{L, M}) \geq ([1-(1-\rho^+)^M]^d)^{2^d},
\end{equation*}

\noindent
and this tends to 1 as $M \rightarrow \infty$ (for fixed $d$) uniformly in $L \geq M$.
\end{proof}

\begin{definition}\label{Mtrimming}
Let $B$ be a box of the form $B_L(x)$. We say that $B$ is \textbf{$M$-good}
if $B$ is $M$-captured and contains no frozen minus vertices. We define
$B[M]$, the \textbf{$M$-trimming} of $B$ as
\begin{equation*}
 B\setminus  \left( \bigcup_{i=1}^{2^d} \bar{C}^i (M) \right),
\end{equation*}
\noindent
where each $\bar{C}^i(M)$ is a cube within $B$ containing exactly $M^d$ vertices
including the i\textsuperscript{th} corner of $B$ - see Figure \ref{fig:Mtrimming}
for the case $d=2$.
\end{definition}

\begin{figure}[h]
\centering
\includegraphics{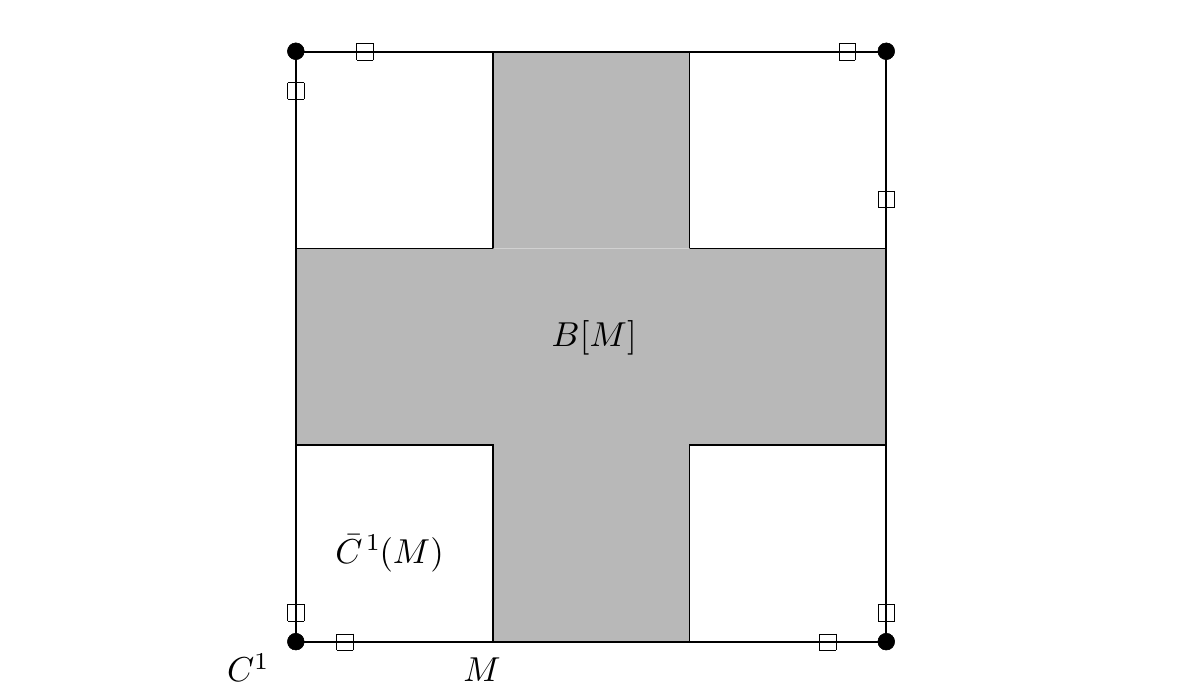}
\caption{$B[M]$ (grey), the $M$-trimming of an $M$-good box $B$.}
\label{fig:Mtrimming}
\end{figure}

\begin{lemma}\label{lemmaprobMgood}
Given $\rho^+ >0$ and $\epsilon >0$, there exist $L<\infty$ and $M<L$ such
that for all sufficiently small $\rho^-$ (depending on $d, L, M, \epsilon$ and
$\rho^+$),
\begin{equation*}
 \mathbb{P} (B_L \text{ is M-good}) > 1-\epsilon.
\end{equation*}
\end{lemma}

\begin{proof}
 By Lemma \ref{lemmaMcaptured}, we may choose $L,M$ with $M < L$ so that

\begin{equation*}
 \mathbb{P} (B_L \text{ is } M\text{-captured}) > 1-\frac{\epsilon}{2}.
\end{equation*}
\noindent
We may also pick $\rho^-$ small enough so that the probability that $B_L$ contains
any frozen minus vertices is less than $\frac{\epsilon}{2}$. Thus
\[
 \mathbb{P} (B_L \text{ is not } M\text{-good}) < \epsilon/2 + \epsilon/2 = \epsilon.
\]
\end{proof}

\begin{lemma}\label{lemmaB[M]}
 Let $M<L$ and $E_{M,L}$ be the event that $B_L$ is $M$-good. Then
 \[
 \mathbb{P}(\text{for any }\sigma'(0), \text{ all vertices in }B_L[M] \text{ fixate to } +1 \mid E_{M,L}) = 1.
 \]
\end{lemma}

\begin{proof}
If $B_L$ is $M$-good, then by Lemma \ref{lemmabootstrap}, almost surely for some time
$t_0$, $\sigma'(t_0)|_{B_L} \equiv +1$. In this case, if $M>0$,
$\sigma'(t)|_{B_L}$ need not stay identically $+1$ for $t>t_0$ because
vertices near the corners can change from $+1$ to $-1$. But a moment's
thought shows that the only vertices near a corner $C^i$ that can change
to $-1$ are those in a subset of the cube $\bar{C}^i(M)$ - see Definition
\ref{Mtrimming}. This is because the frozen plus vertices $C^{i, j}$ (from Definition
\ref{Mcaptured})  protect against the flipping of plus vertices
beyond a rectangular parallelipiped contained in $\bar{C}^i(M)$. Note that because $M<L$,
every vertex in $B_L[M]$ has at least $d+1$ neighbors in $B_L[M]$. For example, for $d=2$ (see
Figure \ref{fig:Mtrimming}), $B_L[M]$ is the union of two rectangles each of width at least
three (hence at least two) so that every vertex in a rectangle has at least three
neighbors in the rectangle. Thus $B_L [M]$ will have
$\sigma'(t)|_{B_L[M]} \equiv +1$ for all $t \geq t_0$. Of course the same
arguments apply to any translated box $B_L(x)=B_L +x $ and to
$B_{L} [M] (x)= B_{L} [M] + x$.
\end{proof}

\subsection{Proofs of main results}

The first of the two theorems follows easily from Lemmas \ref{lemmabootstrap} and
\ref{lemmacaptured}.

\begin{proof}[Proof of Theorem \ref{thm_frozen_+d}]
If the box $B_L$ has all $2^d$ corners frozen plus, and if at some time $t_0$, \\
$\sigma'(t_0) |_{B_L} \equiv +1$,
then $\sigma'(t) |_{B_L} \equiv +1$ for all $t \geq t_0$. This is because after time $t_0$
every vertex in $B_L$ (other than the corners whose spin value is frozen) will have at
least $d+1$ plus neighbors, so it won't flip sign.
If $B_L$ is also captured, then by Lemma \ref{lemmabootstrap}, with
probability one, $\sigma'(t) |_{B_L} \equiv +1$ will occur for some $t$. Now by
Lemma \ref{lemmacaptured}, with probability one, $B_L$ will be captured for some $L$ and
so $\sigma'(0)$ will fixate to $+1$. The same argument can be translated to $x$ and
$B_L(x)$ for any $x\in \mathbb{Z}^d$.
\end{proof}

\begin{proof}[Proof of Theorem \ref{thm_frozen_+-}]
We first introduce an auxiliary graph $G$, with vertex set
$\mathbb{Z}^d$, but where $y_1, y_2$ are neighbors (with edge $\{y_1, y_2 \}$)
if $\| y_1 - y_2 \|_\infty =1$, so every $y$ has $3^d-1$ neighbors. Pick $M<L$ to be determined later and tile $\mathbb{Z}^d$ with boxes $C_L (y)= B_L ((2L+1)y), ~y\in \mathbb{Z}^d$. For disjoint such 
boxes the events of
being $M$-good are independent, so the collection of $M$-good boxes and $M$-bad (that is, not $M$-good) 
boxes defines an independent percolation process on a ``renormalized'' copy 
of $G$, called $\hat G$, by referring to a vertex $y \in \hat G$ as {\bf good} 
if $C_L(y)$ is $M$-good in $G$ and {\bf bad} otherwise. Note that when 
two vertices in $\hat G$ are neighbors, this corresponds (for example, 
for $d=3$), to $C_L(y_1)$ and $C_L(y_2)$ sharing either a
face or an edge or just a corner. (The reason for using this notion of connectedness is that
a standard (using only nearest neighbor edges) cluster of fixed minus and flipping
vertices can extend beyond the standard cluster of $M$-bad (that is, not $M$-good) boxes
into the $G$-cluster of $M$-bad boxes and beyond into the $G$-closure of that $G$-cluster.)

%HANA PLANS TO EXTEND THE NEXT REMARK TO AN ALTERNATE PROOF OF THM~3.4
%\begin{comment}
%\\

%{\bf REMARK}
%Note that for any $d$, if two good boxes share an $d-1$ dimensional "face", i.e. if the corresponding vertices in the new lattice are distance $1$ from each other, then the corresponding $2^{d-1}$ clusters of flipping sites (originally located in the corners of the boxes) will fixate to $+1$ (this follows from the fact that, by attractiveness, the $+1$ values will extend all the way out to the adjoining corners simultaneously thus causing the spins to fixate). 

% (Re-)introducing the notation of $*$-connectedness as being the distance $\sqrt{2}$ from each other observe that for any $d$ the bad boxes need to be $*$ connected to enable the flippers to establish the (regular) connectedness. 
 
%For $d=2$ the numerical improvement on $p^\ast$ is especially notable. Set 
%$p^\ast$ to equal site percolation threshold for $Z^2$. By the result 
%of (???) REF NEEDED since we have a product measure on the new lattice, the bad 
%sites do not percolate as soon as the good ones do.

% For $d>2$ it is enough to take $p^\ast$ small enough so that the bad sites do not $*$-percolate. 

%{\bf END OF REMARK}
%\\

%\end{comment}

Let $\mathcal{C}$ denote the cluster of bad vertices containing the origin in
this independent site percolation model of sites in $\hat G$, and let
$\bar{\mathcal{C}}$ denote its closure (that is, $\mathcal{C}$ unioned with the set of good sites that are
$G$-neighbors of sites in $\mathcal{C}$). We will show that the $G$-cluster $\mathcal{C}^\ast$ containing the origin,
consisting of (eventually) fixed minus vertices together with flipping vertices of $\sigma'$, satisfies
\begin{equation}\label{Cast}
\mathbb{P}\left( \mathcal{C}^\ast \subseteq \bigcup_{y\in \bar{\mathcal{C}}} C_L (y) \text{ for every } \sigma'(0) \right)=1.
\end{equation}

Once Equation~(\ref{Cast}) has been verified, the proof is completed as follows.
By standard percolation
arguments, there is some $p^\ast >0$ (one can take, for instance,
$p^\ast = 1 / (3^d-1)$, since $3^d-1$ is the number of
neighbors of any vertex in $G$), such that, if

\begin{equation} \label{probbadsite}
\mathbb{P} ( y \text{ is bad}) = \mathbb{P}(B_L \text{ is not }M\text{-good})<p^\ast,
\end{equation}

\noindent
then there is almost surely no percolation of bad sites and
$\mathbb{E}(|\bar{\mathcal{C}}|) < \infty$.
To finish the proof we use Lemma \ref{lemmaprobMgood} to choose $\rho^-$ so small that
inequality (\ref{probbadsite}) is valid, and note that by inequality (\ref{Cast}),
\begin{equation*}
| \mathcal{C}^\ast| \leq |\bar{\mathcal{C}}| (2L+1)^d < \infty.
\end{equation*}
\end{proof}

It remains to prove~(\ref{Cast}), for which we review some old and and  
introduce some  new notation. 
We refer to vertices $x$ and $y$ in (the original) $\mathbb{Z}^d$ as $G$-neighbors
if $\|x-y\|_{\infty}=1$ (and we view them as vertices in the graph $G$).
We refer to two boxes  $C_L(y_1)$ and $C_L(y_2)$ as $G$-box-neighbors 
if $y_1$ and $y_2$ are $G$-neighbors in (the renormalized) $\mathbb{Z}^d$ (and we view these vertices in the graph $\hat G$).
$G$-clusters and $G$-box-clusters (of certain vertices and boxes) and their
boundaries will be used later; their definitions are analogous.
Note that there exists a $G$-path connecting a vertex in $C_L(y_1)$ to one in 
$C_L(y_2)$ and contained within $C_L(y_1) \cup C_L(y_2)$  
iff  $C_L(y_1)$ and $C_L(y_2)$  are  $G$-box-neighbors.  

A corner region $R(y)$ of a box $C_L(y)$ is a cube entirely within $C_L(y)$
containing exactly $M^d$ vertices including one of the $2^d$ corner vertices
of $C_L(y)$. (A vertex is  a corner vertex if it has only $d$ neighbors 
within $C_L(y)$). Any box $C_L(.)$ contains exactly $2^d$ distinct corner regions 
and their $G$-boundaries are disjoint when $M<L$. 

We call a box {\bf good} if it is $M$-good. If a box is not good, it is {\bf bad}. 
We call a vertex bad if it is either
%\begin{itemize}
in a bad box or 
in one of the corner regions of a good box; otherwise, it is called a good vertex. Note that by Lemma~\ref{lemmaB[M]}, all good vertices fixate to $+1$.
%\end{itemize}

\begin{lemma} \label{GdBxPth}

Let $x_1 \in R_1(x)$ and $x_2 \in R_2(x)$ where $R_1(x)$ and $R_2(x)$ 
are two distinct corner regions of a good $C_L(x)$.
Let $\gamma$ be a $G$-path of bad vertices connecting $x_1$ and $x_2$;
then $\gamma$ contains a vertex $z$ such that $z \notin C_L(x)$.

%{\bf OR:}
%
%There is no $G$-path of bad sites that is contained entirely within a good box and that contains vertices that belong to more than one corner region.
%
%{\bf OR:}
%
% Any $G$-path of bad vertices that is contained entirely within a good box is also contained in a single corner region (of that box)).
%
\end{lemma}

\begin{proof}
By way of contradiction, suppose such a $G$-path $ \gamma $ exists in $C_L(x)$.
After translating $C_L(x)$ to $\{1,2,\dots,2L+1\}^d$, we may assume that $\gamma$  starts at a vertex $\gamma _1 = (a_1, a_2,\ldots, a_d)$ 
such that for all $i$, $a_i \leq M$.  If $\gamma$ ends in a different 
corner region, %$R(x)$, \
denote the last vertex in $\gamma$  by $\gamma_2=(b_1, b_2, \ldots, b_d)$ and note that 
at least for one $i$, $b_i> M+3$. Without loss in generality, assume $b_1 >M+3$ and note that the 
$G$-neighbors  connecting $\gamma_1$ and $\gamma_2$ differ in any coordinate
by at most $1$. Thus there exists a bad vertex $\gamma_3=(c_1, c_2, ...c_d)$ 
in $\gamma$ with $c_1=M+1$. Recall, 
however, that in a good box only the corner regions contain bad 
vertices, so that any site $x=(x_1, x_2, \ldots, x_d)$ 
with $M+1 \leq x_i \leq M+3$ for at least one $i$, is good. 
\end{proof}

\begin{definition} \label{Adj}
For $x\neq y$, two corner regions, $R(x)$ and $R(y)$, of boxes $C_L(x)$ and $C_L(y)$  are called adjacent if there exists $v \in R(x)$ and $v' \in R(y)$ such that $v$ and $v'$ are $G$-neighbors.
\end{definition}

The proof of the next lemma is straightforward and is left to the reader.

\begin{lemma} \label{ADJ}
Let $x,y$ be distinct vertices in  $\mathbb{Z}^d$.
\begin{itemize}

\item{}If there is a $G$-path connecting $R(x)$ and $C_L(y)$ and remaining in $C_L(x) \cup C_L(y)$ then $C_L(y)$ 
has a corner region $R(y)$ adjacent to $R(x)$.
\item{} If $R(x)$ is adjacent to $R(y)$ and $R(y)$ is adjacent to $R(z)$ 
then $R(x)$ is adjacent to $R(z)$ (transitivity of the adjacency relation).
\end{itemize}
\end{lemma}

%The proof of Lemma \ref{ADJ} consists of checking the condition 
%of Definition \ref{Adj} and is left to the reader.

%Let $C_L(x)$ be a good box and $C_L(y)$ be a bad box and assume that 
%$C_L(y)$ is a member of a bad cluster containing the origin and $C_L(x)$ 
%is a good box that is a $G$-box neighbor of $C_L(y)$. 
 
\begin{lemma} \label{AdjCorPth}
Let $C_L(y)$ be a good box and $C_L(x)$ be a bad one. 
Assume that %$C_L(x)$ is in a $G$-box-cluster of bad boxes containing $C_L(0)$ and that
$C_L(y)$ is a G-box-neighbor of $C_L(x)$. 
Let $\gamma$ be a $G$-path of bad sites starting in $R(y)$, where
$R(y)$ is adjacent to a corner region of $C_L(x)$.
%and let $C_L(x)$ be a 
%bad $G$-box-neighbor of $C_L(y)$ such that $C_L(x)$ has a corner region 
%adjacent to $R(y)$. 
Then $\gamma$ can only exit $C_L(y)$ into either
\begin{enumerate}
\item a bad box that is a $G$-box-neighbor of $C_L(x)$, or
\item a corner region, adjacent to $R(y)$, of a good box that is a 
$G$-box-neighbor of $C_L(x)$.
\end{enumerate}
\end{lemma}
\begin{proof}
If the path exits $R(y)$ into a good box $C_L(z)$, then it must (by the reasoning 
of Lemma~\ref{GdBxPth}) enter into a corner region that is adjacent to $R(y)$ and
hence by Lemma~\ref{ADJ}, $C_L(z)$ is a $G$-box-neighbor of both $C_L(y)$ and
$C_L(x)$. If it exits into a bad box, $C_L(z')$, 
then it must enter into an adjacent corner region
$R(z')$ or into its $G$-boundary, but in either case, $C_L(z')$ is again a 
$G$-box-neighbor of both $C_L(y)$ and $C_L(x)$.
\end{proof}

%NEXT IS THE OLD PROOF (TO BE DELETED IF NEW PROOF IS OK):

%THIS PROOF NEEDS SOME MODIFICATIONS (CN)
%We first establish that since $C_L(y)$ is good, then $\gamma$ must exit $R(y)$ 
%into an adjacent $R(z)$. 
%I DON'T THINK LAST STATEMENT IS CORRECT, BUT IT'S CLOSE.
%The result is clear in case $C_L(x)$ and $C_L(z)$ only share a 
%hyperplane of dimension $0$ (touch at the corner). If they share a hyperplane 
%of dimension $1$ or higher, then there exists a vertex 
%$x_1=\{a_1, a_2,.... a_d\}\in Ryx)$ and a vertex  $z_1=\{b_1, b_2,.... b_d\}\in R(z)$ 
%such that (WLOG) $a_1=b_1=M$. Suppose $x_1 \in \gamma$ and note that 
%$u=\{M+1, .,...\}\in C_L(z)$, but not in $R(z)$. However, by reasoning of 
%Lemma \ref{GdBxPth} $u$ is a good site, so $u \notin \gamma$.

%Now $\gamma$ can also exit $R(y)$ into a  bad box $C_L(s)$ if $C_L(s)$ 
%has a corner region adjacent to $R(y)$ (by the discussion in the previous 
%paragraph). But that means that $C_L(x)$ and $C_L(s)$ share a corner by 
%the transitivity of adjacency relation. Hence they belong to the same 
%bad cluster.
%END OF OLD PROOF.

The next proposition is an immediate consequence of the previous lemma
in the setting where $C_L(x)$ is in a $G$-box-cluster of bad boxes.

\begin{proposition}

There is no $G$-path of bad sites from a $G$-box 
cluster of bad boxes that exits the union of the $G$-box cluster and 
its $G$-box-boundary.

\end{proposition}
%{\bf End}
%\noindent
%By standard percolation
%arguments, there is some $p^\ast >0$ (one can take, for instance, 
%$p^\ast = 1 / (3^d-1)$, since $3^d-1$ is the number of
%neighbors of any vertex in $G$), such that, if
%
%\begin{equation} \label{probbadsite}
% \mathbb{P} ( y \text{ is bad}) = \mathbb{P}(B_L \text{ is not }M\text{-good})<p^\ast,
%\end{equation}
%
%\noindent
%then there is almost surely no percolation of bad sites and 
%$\mathbb{E}(|\bar{\mathcal{C}}|) < \infty$.
%To finish the proof we use Lemma \ref{lemmaprobMgood} to choose $\rho^-$ so small that
%inequality (\ref{probbadsite}) is valid, and note that by inequality (\ref{Cast}),
%\begin{equation*}
% | \mathcal{C}^\ast| \leq |\bar{\mathcal{C}}| (2L+1)^d < \infty.
%\end{equation*}
%
%\end{proof}

\begin{proof}[Proof of Theorem 3.4]
\noindent
The proof uses the version of Proposition~\ref{schonmann} for Schonmann's
{\bf modified basic version} of bootstrap percolation as discussed in
Example~3 of~\cite{S}.
We will focus on the case $d-1=2$. The case $d-1=1$ is easier and can
be handled without the use of bootstrap percolation. We leave it for the reader
to check that the proof for the cases with $d-1 \geq 3$  proceed essentially
the same as for $d-1=2$.

For $d-1=2$, we first partition  $\mathbb{Z}^2\times \{0, 1, 2, ...K\}$  into disjoint
$2\times 2\times (K+1)$ pillars \\
$\mathcal P_{i,j}= \{ (2i, 2j), (2i+1, 2j), (2i, 2j+1), (2i+1, 2j+1)\}\times\{0, 1, 2, ..., K\}$ for $(i,j) \in Z^2$.
If at any time, all vertices in $\mathcal P_{i,j}$ are $+1$, then
they stay $+1$ after that since the bottom layer is frozen to $+1$,
each site in the top layer has 3 neighbors (within the pillar) out of
its 5 total neighbors equal to $+1$, and all other sites have 4 neighbors
(within the pillar) out of its 6 total neighbors equal to $+1$.

A pillar $\mathcal P_{i,j}$ is fixed $+1$ in this way at time zero, with
probability $\theta^{4K}$ and the set of such $(i,j)$ are chosen independently
of each other. We can now apply the modified Proposition~\ref{schonmann} as long as
we can show that if a pillar $\mathcal P_{i,j}$ (at time $t$) has at
least two neighboring pillars -- one in each coordinate direction -- all
$+1$ (at time $t$), then almost surely by some random time $t+T$,
$P_{i, j}$ will have become all $+1$.

This last claim is verified by first arguing, like in the proof of
Lemma~\ref{lemmaboostrapweak}, there is strictly positive probability that $\mathcal P_{i,j}$
will be all $+1$ by the time $t+1$, and then proceeding as in the proof of Lemma~\ref{lemmabootstrap}.
To explain the first argument, suppose the all $+1$
neighboring pillars are $\mathcal P_{i-1,j}$ and $\mathcal P_{i,j-1}$.
Then the vertices in the first layer (above the
frozen sites), $(2i, 2j, 1), (2i+1, 2j, 1), (2i, 2j+1, 1), (2i+1, 2j+1, 1)$
can flip to $+1$ (if they are not already $+1$) in that order followed by
the sites in layers $2, 3, \ldots, K$ in that order.
\end{proof}

\bigskip

{\bf Acknowledgments.}
The authors thank Leo T. Rolla for many fruitful discussions. The research reported in this paper was supported in part by NSF grants 
DMS-1007524 (S.E. and C.M.N.), DMS-1419230 (M.D.) and OISE-0730136 (S.E.,
H.K. and C.M.N.). V.S. was supported by ESF-RGLIS network and  by
Brazilian CNPq grants 308787/2011-0 and 476756/2012-0 and FAPERJ grant
E-26/102.878/2012-BBP.

\end{document}